\documentclass[oneside,12pt]{amsart}
\usepackage{amssymb,amscd,amsmath,latexsym,amsthm}
\usepackage[mathcal]{euscript}
\usepackage{diagrams}

\begin{document}

\title {Twisted equivariant K-theory, groupoids and proper actions}

\author[Jose Cantarero]{Jose Cantarero}
\address{Department of Mathematics
University of British Columbia, Vancouver, B.C., Canada}
\email{cantarer@math.ubc.ca}

\def \P{{\mathbb P}}
\def \C{{\mathbb C}}            
\def \N{{\mathbb N}} 
\def \Z{{\mathbb Z}}
\def \R{{\mathbb R}}
\def \F{{\mathbb F}}
\def \G{{\mathcal G}}
\def \H{{\mathcal H}}
\def \K{{\mathcal K}}
\def \Fa{{\mathcal F}}
\def \Hy{{\mathbb H}}

\begin{abstract}

In this paper we define twisted equivariant $K$-theory for actions of Lie groupoids. For a Bredon-compatible Lie groupoid 
$\G $, this defines a periodic cohomology theory on the category of finite $\G $-CW-complexes with $\G$-stable projective 
bundles. A classification of these bundles is shown. We also obtain a completion theorem and apply these results to proper 
actions of groups.

\end{abstract}

\maketitle

\newtheorem{thm}{Theorem}[section]
\newtheorem{cor}[thm]{Corollary}
\newtheorem{lemma}[thm]{Lemma}
\newtheorem{prop}[thm]{Proposition}
\newtheorem{conj}[thm]{Conjecture}
\theoremstyle{definition}
\newtheorem{defn}[thm]{Definition}
\newtheorem{example}[thm]{Example}



\section{Introduction}

Complex equivariant $K$-theory for actions for Lie groupoids was defined in \cite{C} by using a special class of vector bundles, 
called extendable vector bundles. This defines a cohomology theory on the category of $\G$-spaces. If we restrict our attention
to finite $\G$-CW-complexes and Bredon-compatible Lie groupoids, then we have enough extendable vector bundles to obtain 
equivariant $K$-theory when the action is given by a Lie group. Moreover, in this case the cohomology theory turns
out to be periodic as well.

An interesting feature of these cohomology theories is the invariance under weak equivalence.  
Since $\G$-CW-complexes are made out of cells for which the action of the groupoid is equivalent to the action of a compact
Lie group on a finite $G$-CW-complex, we are able to use results of classical equivariant $K$-theory to prove analogous 
properties for groupoids. 

Atiyah and Segal twist equivariant $K$-theory for actions of a compact Lie group $G$ using $G$-stable projective
bundles \cite{AS2}. Since stable projective bundles and sections behave well under weak equivalences, it seems natural to use 
 $\G$-stable projective bundles, which can be defined in an similar way, to twist $\G$-equivariant $K$-theory.

\begin{defn}

Let $\G $ be a Lie groupoid and $X$ a $\G $-space. A $\G$-projective bundle on $X$ is a $\G $-space $P$ with a $\G$-equivariant map
$ p : P \longrightarrow X $ such that there exists an equivariant open covering $ \{ U_i \} $ of $X$ for which $ P_{|U_i} =
U_i \times _{\pi } \P (E) $ for some $\G $-Hilbert bundle $E$ on $G_0 $. Moreover, We shall call $P$ a $\G$-stable projective 
bundle if $ P \cong P \otimes \pi^*\P (U(\G )) $ for some locally universal $\G$-Hilbert bundle $U(\G )$ on $G_0$.

\end{defn}

$G$-equivariant $K$-theory can be represented by a space of Fredholm operators on a $G$-stable Hilbert space. This is used
to construct twisted $K$-theory. For a $G$-stable projective bundle, we can consider a suitable bundle of Fredholm operators 
associated to it and define twisted $K$-theory with the sections of this bundle \cite{AS2}. For actions of groupoids, we have 
a representability theorem if we consider not all maps into such a space of Fredholm operators, but only those which are 
extendable:

\begin{defn}

Let $H$ be a stable representation of $\G$, that is, a $\G$-Hilbert bundle on $G_0$ such that $ H \oplus U(\G) = H $. We say 
that a $\G$-equivariant map $f: X \to Fred'(H) $ is extendable if there is another $\G$-equivariant map 
$g: X \to Fred'(H) $ such that $gf = v \pi _X $ for some section $v$ of $Fred'(H) \to G_0 $, where $ \pi _X : X \to G_0 $
is the anchor map.

\end{defn}

\begin{thm}

Let $H$ be a stable representation of a Bredon-compatible finite Lie groupoid $\G$. Then:
\[ K_{\G}(X) = [ X , Fred'(H) ]_{\G}^{ext} \]

\end{thm}

Choosing all sections of the Fredholm bundle corresponds to choosing all vector bundles in the untwisted case. To make 
these new theories an extension of untwisted $K$-theory, we need to consider extendable sections. Then we can define twisted 
$\G$-equivariant $K$-theory as the group of extendable homotopy classes of extendable sections of a suitable Fredholm bundle, 
that is homotopy classes where the homotopies run over extendable sections. Extending it to all degrees as in \cite{AS2}, we 
obtain a cohomology theory:

\begin{thm}

If $\G $ is a Bredon-compatible finite Lie groupoid, the groups ${^PK_{\G }^n(X)} $ define a $ \Z /2 $-graded cohomology 
theory on the category of finite $\G$-CW-complexes with $\G$-stable projective bundles, which is a module over untwisted
$\G$-equivariant $K$-theory.  

\end{thm}

In the untwisted case, under some conditions, we have a completion theorem that relates the completion of $\G$-equivariant $K$-theory with respect
to the augmentation ideal $I_{\G}$ to the non-equivariant $K$-theory of the Borel construction. The recent completion theorem
for twisted equivariant $K$-theory for actions of compact Lie groups in \cite{La} provides the necessary results to use 
induction over cells. In the twisted case, however, the completion theorem will relate the completion of twisted $\G$-equivariant 
$K$-theory of $X$ with respect to $I_{\G}$ to the twisted $\G$-equivariant $K$-theory of $ X \times _{\pi} E\G$. Our 
generalization of the completion theorem in this context is the following theorem: 

\begin{thm}
Let $\G$ be a Bredon-compatible finite Lie groupoid, $X$ a finite $\G$-CW-complex and $P$ a $\G$-stable projective bundle on $X$. 
Then we have an isomorphism of $K_{\G}^*(G_0)$-modules: 
\[ {^PK_{\G }^n (X)_{I_{\G}}^{\wedge}} \longrightarrow {^{P \times _{\pi} E\G}K_{\G }^n (X \times _{\pi} E\G)} \]
\end{thm}

When $S$ is a Lie group, not necessarily compact, we can define twisted equivariant $K$-theory for proper actions of $S$ using
the groupoid $ S \rtimes \underline{E}S $, where $ \underline{E}S $ is the universal space for proper actions of $S$. While in 
general these actions are not Bredon-compatible, there some interesting cases in which they are. Some examples are given by 
proper actions of discrete groups, pro-discrete groups, almost compact groups and matrix groups. They are discussed in more 
detail in section 5.
\newline

The category of $\G$-orbits behaves similarly to the corresponding category for a compact Lie group. In particular, we are
able to use some of these properties to prove an analogue of Elmendorf's construction \cite{E}. This construction is the
key to the classification of $\G$-stable projective bundles. In fact, we show that isomorphism classes of $\G$-stable projective 
bundles over $X$ are classified by $ H_{\G}^3(X) $. This is done by constructing a particular model for the space that 
represents  $ H_{\G}^3(-) $ which admits a natural $\G$-stable projective bundle on it.

\section{Representable $K$-theory}

\noindent In this paper all groupoids are Lie groupoids.

\begin{defn}

Let $\G $ be a Lie groupoid and $X$ a $\G $-space. A $\G$-Hilbert bundle on $X$ is a $\G $-space $E$ with an equivariant map
$ p : E \longrightarrow X $ which is also a locally trivial Hilbert bundle with a continuous linear $\G$-action.

\end{defn}

\begin{defn}

A universal $\G$-Hilbert bundle on $X$ is a $\G$-Hilbert bundle $E$ such that
for each Hilbert bundle $V$ on $X$ there exists a $\G$-equivariant unitary
embedding $ V \subset E $.

\end{defn}

\begin{defn}

A locally universal $\G$-Hilbert bundle on $X$ is a $\G$-Hilbert bundle $E$ such that
there is a countable open cover $ \{ U_i \} $ of $X$ such that $ E_{|U_i}$ is 
a universal $\G$-Hilbert bundle on $U_i$.

\end{defn}

\begin{defn}

A local quotient groupoid is a groupoid $\G$ such that $  G_0  $ admits
a countable open cover $ \{ U_i \} $ with the property that $ \G \rtimes U_i $
is weakly equivalent to an action groupoid corresponding to the proper action of a compact 
Lie group $G$ on a finite $G$-CW-complex.

\end{defn}

\begin{cor}

A finite Lie groupoid is a local quotient groupoid.

\end{cor}

\begin{prop}

If $\G$ is a local quotient groupoid, then there exists a locally universal $\G$-Hilbert bundle on $G_0$ that is unique
up to unitary equivalence.

\end{prop}

\begin{proof}

See \cite{FHT}
\end{proof}

\begin{cor}

If $\G$ is a finite Lie groupoid, then there exists a locally universal $\G$-Hilbert bundle on $G_0$ that is unique up
to unitary equivalence. We denote it by $U(\G)$.

\end{cor}

\begin{prop}

Suppose that $ F : \G \longrightarrow \H $ is a local equivalence. Then the pullback functor
\[ F^* : \{ \text{Locally universal $\H$-Hilbert bundles on } H_0 \} \rightarrow \{ \text{Locally universal $\G$-Hilbert bundles on } G_0 \} \]
is an equivalence of categories.

\end{prop}

\begin{proof}

See \cite{FHT}
\end{proof}

Let $H$ be a stable representation of $\G$, that is, a $\G$-Hilbert bundle on $G_0$ such that $ H \oplus U(\G) = H $. Consider
the associated bundle of Fredholm operators $Fred(H)$ on $G_0$, and the subbundle $Fred'(H)$ of operators $A$ for which the
action $g \to gAg^{-1} $ is continuous. See \cite{EM} for more details on the correct topology for this space.

\begin{defn}

We say that a $\G$-equivariant map $f: X \to Fred'(H) $ is extendable if there is another $\G$-equivariant map 
$g: X \to Fred'(H) $ such that $gf = v \pi _X $ for some section $v$ of $Fred'(H) \to G_0 $, where $ \pi _X : X \to G_0 $
is the anchor map.

\end{defn}

\begin{defn}

We say that a homotopy $ H : X \times I \to Fred'(H) $ of $\G$-equivariant maps is extendable if each $H_t$ is an
extendable $\G$-equivariant map $ X \to Fred'(H) $.

\end{defn}

\begin{defn}

Let $X$ be a $\G$-space, $H$ a stable representation of $\G$ and $ n \geq 0 $. Define the $\G$-equivariant representable $K$-theory groups of $X$ to be 
\[ RK_{\G}^{-n}(X) = [ X , \Omega ^n Fred'(H)]_{\G}^{ext} \]
where this notation denotes the extendable homotopy classes of extendable $\G$-maps. For $\G$-pairs $(X,A)$, define
\[ RK^{-n}_{\G } (X,A) = Ker [ RK^{-n}_{\G } ( X \cup _A X) \stackrel{j^*_2}{\longrightarrow} RK^{-n}_{\G } (X) ] \] 
where $ j_2 : X \rightarrow X \cup _A X $ is one of the maps from $ X$ to the pushout.

\end{defn}

\noindent We could have defined the extendable $K$-groups as in \cite{LO}:
\[ RK^{-n}_{\G } (X) = Ker [ RK_{\G }(X \times S^n) \stackrel{i*}{\longrightarrow} RK_{\G }(X) ] \]
where $ RK_{\G }(X) = [ X , Fred'(H)]_{\G}^{ext} $ and $ i : X \rightarrow X \times S^n $ is the inclusion given by fixing 
a point in $ S^n $. Both definitions are clearly equivalent.

We will now prove this defines a cohomology theory on the category of $\G$-spaces. Since the definition is given by 
homotopy classes of maps, the following follows from the definition.

\begin{cor}

If $f_0$, $f_1 : (X,A) \longrightarrow (Y,B) $ are $\G$-homotopic $\G$-maps between $\G$-pairs, then
\[  f_0^*=f_1^* : RK_{\G}^{-n}(Y,B) \longrightarrow RK_{\G}^{-n}(X,A)   \]
for all $n \geq 0 $.

\end{cor}

The following lemma follows easily from the definitions:

\begin{lemma} 

Let $(X,A)$ be a $\G$-pair. Suppose that $\displaystyle X = \mathop{\coprod } _{i \in I} X_i $, the disjoint union of open $\G$-invariant 
subspaces $X_i$ and set $ A_i = A \cap X_i $. Then there is a natural isomorphism
\[  RK_{\G}^{-n}(X,A) \longrightarrow \prod _{i \in I} RK_{\G}^{-n}(X_i,A_i)  \]

\end{lemma}

\begin{lemma}

Let $\phi : X \longrightarrow Y $ be a $\G$-equivariant map, $H$ be a stable representation of $\G$ and let 
$s: X \longrightarrow Fred'(H) $ be a $\G$-extendable map. Then, there is a $\G$-extendable map  
$t: Y \longrightarrow Fred'(H) $ such that $ s's=t \phi $ for some $\G$-extendable map $s: X \longrightarrow Fred'(H) $.

\end{lemma}

\begin{proof}
Since $s$ is extendable, there is a $\G$-extendable map $s':X \longrightarrow Fred'(H) $ such that $ s's= v = \pi _X = 
v \pi _Y \phi $ for some section $ v : G_0 \longrightarrow Fred'(H)$. Choose $ t = v \pi _Y $. This is a $\G$-extendable
map for it is the pullback of a $\G$-extendable map and we have $ s's = t \phi $.
\end{proof}

\begin{lemma}

Let
\diagram
A & \rTo^{i_1} & X_1 \\
\dTo_{i_2} & & \dTo_{j_1} \\
X_2 & \rTo^{j_2} & X \\
\enddiagram
be a pushout square of $\G$-spaces, $H$ a stable representation of $\G$ and $i_1$ a cofibration. Let 
$s_k : X_k \longrightarrow Fred'(H)$ be $\G$-extendable maps for $ k = 1,2 $ such that $ s_1 i_1 $ and $ s_2 i_2 $ are 
$\G$-extendable homotopic maps from $ A $ to $Fred'(H) $. Then, there is a $\G$-extendable map 
$ t : X \longrightarrow Fred'(H) $ such that $ t j_k $ is $\G$-extendable homotopic to $ s_k $ for $ k = 1,2 $.

\end{lemma}

\begin{proof}

Let $ F : A \times I \longrightarrow Fred'(H) $ be a $\G$-extendable homotopy with $ F_0 = s_1 i_1 $ and $ F_1 = s_2 i_2 $. 
There is $ v_2 : G_0 \longrightarrow Fred'(H) $ such that $ s_2's_2 = v \pi _2 $ for some $ s_2' : X_2 \longrightarrow Fred'(H) $. 
By the previous lemma, there is a $\G$-extendable homotopy $ \bar{F} : G_0 \times I \longrightarrow Fred'(H) $ such that
$ F'F = \bar{F} \circ (\pi _A \times id) $ for some $ F' : A \times I \longrightarrow Fred'(H) $. We can also make it
satisfy $F'_1 = s'_2 i_2 $ by multiplying by a convenient constant homotopy for $G_0$.
 
Now, since $ A \times I \stackrel{i_1 \times id}{\longrightarrow} X_1 \times I $ is a $\G$-equivariant cofibration, there are 
$ G,G' : X_1 \times I \longrightarrow Fred'(H) $ that extend $ F $ and $ F' $ respectively. Therefore $ G'G $ must be an 
extension of $\bar{F} \circ (\pi _A \times id) $ to $ X_1 \times I $. In fact, by the previous lemma, we can choose $G$ and 
$G'$ so that $ G'G = \bar{F} (\pi _1 \times id) $. Therefore $G$ is a $\G$-extendable homotopy.

Let $ G_0 = s_1 $ and $ G_1 = \tilde{s}_1 $. The extendable $\G$-homotopy classes of these two maps are equal, 
and $ \tilde{s}_1 i_1 = G_1 i_1 = F_1 = s_2 i_2 $. So we can easily extend this to a map $ t : X \longrightarrow Fred'(H) $ 
such that $ t j_1 = \tilde{s}_1 $ and $ t j_2 = s_2 $. Therefore $ t j_k $ is $\G$-extendable homotopic to $s_k$ for $k=1,2$. 
In fact, it is given by:
\[ t(x) = \left \{ \begin{array}{cc}
\tilde{s}_1(x_1) & \text{if $x = j_1(x_1)$} \\
s_2(x_2) & \text{if $x = j_2(x_2)$} \end{array} \right. \]
We only need to prove that $t$ is extendable. Let $\tilde{s_1}' = G'_1 $, $s_2' i_2 = F'_1 $. We have
$ \tilde{s_1}' i_1 = G'_1 i_1 = F'_1 = s_2' i_2 $. Consider
\[ t'(x) = \left \{ \begin{array}{cc}
\tilde{s_1}'(x_1) & \text{if $x = j_1(x_1)$} \\
s_2'(x_2) & \text{if $x = j_2(x_2)$} \end{array} \right. \]
Let $ \bar{F}_1 = v_1 $. Then we have $ \tilde{s_1}' \tilde{s}_1 = v_1 \pi _1 $. Now consider the map:
\[ v(x) = \left \{ \begin{array}{cc}
v_1(\pi_1 (x_1)) & \text{if $x = \pi_1(x_1)$} \\
v_2(\pi_2 (x_2)) & \text{if $x = \pi_2(x_2)$} \end{array} \right. \]
This map is well defined. If $ \pi_1(x_1) = \pi_2(x_2) $, then $ x_1 = i_1 a $ and $ x_2 = i_2 a $, and
$ v_1(\pi _1 (x_1)) = v_1 \pi _A a = \bar{F} _1 \pi _A a = F_1 \pi _A a = F_1 \pi _2 (x_2) = v_2(\pi _2 (x_2)) $. It is a routine check that 
$ t't = v \pi _X $, thus $ t $ is extendable
\end{proof}

\begin{lemma}

Let
\diagram
A & \rTo^{i_1} & X_1 \\
\dTo_{i_2} & & \dTo_{j_1} \\
X_2 & \rTo^{j_2} & X \\
\enddiagram
be a pushout square of $\G$-spaces and $i_1$ a cofibration. Then there is a natural exact sequence, infinite to the left
\[  \ldots \stackrel{d^{-n-1}}{\longrightarrow} RK_{\G}^{-n}(X) \stackrel{j_1^* \oplus j_2^*}{\longrightarrow} RK_{\G}^{-n}(X_1) \oplus
    RK_{\G}^{-n}(X_2) \stackrel{i_1^*-i_2^*}{\longrightarrow} RK_{\G}^{-n}(A) \stackrel{d^{-n}}{\longrightarrow} \ldots \]
\[  \ldots \longrightarrow RK_{\G}^{-1}(A) \stackrel{d^{-1}}{\longrightarrow} RK_{\G}^0(X) \stackrel{j_1^* \oplus j_2^*}{\longrightarrow} 
    RK_{\G}^0(X_1) \oplus RK_{\G}^0(X_2) \stackrel{i_1^*-i_2^*}{\longrightarrow} RK_{\G}^0(A) \] 

\end{lemma}

\begin{proof}
It is a consequence of the two previous lemmas, the results in \cite{B} and the proof of lemma 3.8 in \cite{LO}.
\end{proof}

For any stable representation $H$ of $\G$ there is a $\G$-map $ \Omega ^n Fred'(H) \to \Omega ^{n+2} Fred'(H) $, which therefore
induces a Bott map $ b(X) : RK_{\G}^{-n}(X) \rightarrow RK_{\G}^{-n-2}(X) $. By the definition of the relative groups, we also
have Bott maps $ b(X,A) : RK_{\G}^{-n}(X,A) \rightarrow RK_{\G}^{-n-2}(X,A) $. We will prove that these maps are isomorphisms 
for finite $\G$-CW-complexes.

\begin{prop}

Suppose that $ F : \G \longrightarrow \H $ is a local equivalence. Then the pullback functor induces an homeomorphism:
\[ F^* : \{ \text{Sections of a fiber bundle on } \H \} \longrightarrow \{ \text{Sections of the pullback fiber bundle on } \G \} \]
\end{prop}

\begin{proof}

Assume we have a local equivalence $\G \to \H$. Given a section $v$ of a fibre bundle $ P \to H_0$, we can consider the 
section $ F^*(v) : G_0 \to F^*(P) $ defined by $ F^*(v)(x) = (x,v(F(x)))$. And on the other hand, given a section $w$ of a fibre 
bundle $ Q \to G_0 $, we can consider the section $ F_*(w) : H_0 \to F_*(Q) $ defined by $ F_*(t)(x) = (x, w(y)) $
where $ y \in G_0 $ is such that there is $ h \in H_1 $ that satisfies $ F(y) = s(h) $ and $t(h)=x$.
\end{proof}

\begin{lemma}

Suppose that $ F : \G \longrightarrow \H $ is a local equivalence. Then we have an isomorphism:
\[  F^* : RK_{\H }^*(H_0) \longrightarrow RK_{\G }^*(G_0)  \]
\end{lemma}

\begin{proof}

\noindent $ RK_{\H }^{-n}(H_0) = [ H_0 , \Omega ^n Fred'(H)]_{\H}^{ext} = $
\newline

\noindent $ = \text{Extendable $\H $-homotopy classes of extendable sections of $ \Omega ^n Fred'(H) $} \cong $
\newline

\noindent $ \cong \text{Extendable $\G $-homotopy classes of extendable sections of $ F^*(\Omega ^n Fred'(H)) $} =  $
\newline

\noindent $ = \text{Extendable $\G $-homotopy classes of extendable sections of $\Omega ^n Fred'(F^*H)$} = $
\newline

\noindent $  = [ G_0 , \Omega ^n Fred'(F^*H)]_{\G}^{ext} = $
\newline

\noindent $ = RK_{\G }^{-n}(G_0) $
\end{proof}

\begin{cor}

If $ \G $ and $ \H $ are weakly equivalent, we have an isomorphism  $ RK_{\H }^*(H_0) \cong RK_{\G }^*(G_0) $

\end{cor}

\begin{cor}

If $\G$ is a Bredon-compatible finite Lie groupoid and $U$ is a $\G$-cell, then $RK_{\G}^*(U) \cong K_G^*(M) $ for some compact Lie group
$G$ and a finite $G$-CW-complex $M$.

\end{cor}

\begin{proof}
Since $U$ is a $G$-cell, we know that $ \G \rtimes U $ is weakly equivalent to $ G \rtimes M $ for some compact Lie group
$G$ and a finite $G$-CW-complex $M$. Therefore, by the previous corollary:
\[ RK_{\G \rtimes U }^*(U) \cong RK_{G \rtimes M }^*(M) \]
Let $H$ be a locally universal representation of $\G$. We want to see that $ \pi _U^*(H) = U \times _{\pi} H $ is a locally universal
$ \G \rtimes U $-Hilbert bundle. Notice that if $U$ is a $\G$-cell, so is any open $\G$-subspace of $U$. Therefore it is 
enough to prove the previous assertion with universal Hilbert bundles. So assume $H$ is a universal $\G$-Hilbert bundle.

Now let $V$ be a $\G \rtimes U $-vector bundle on $U$. This a $\G $-vector bundle on $U$, and since $\G$ is Bredon-compatible, there 
is a $\G$-vector bundle $W$ on $G_0$ such that $ \pi _U^* (W) = V \oplus V'$ for some other $\G$-vector bundle $V'$ on $U$. Since
$H$ is universal, there is a unitary $\G$-embedding $ W \hookrightarrow H $ and so $ \pi _U^*(W) \hookrightarrow \pi _U^*(H) = 
U \times _{\pi} H $. Since $V$ is a direct summand of $ \pi _U^* (W) $, we have a unitary $\G \rtimes U $-embedding
$ V \hookrightarrow U \times _{\pi} H $.

Thus, if $E$ is a locally universal Hilbert representation of $G$, then $ E \times M $ is a locally universal
$G \rtimes M $-Hilbert bundle.
\newline

\noindent $ RK_{\G \rtimes U }^{-n}(U) = [ U , \Omega ^n Fred'(U \times _{\pi} H) ]_{\G \rtimes U}^{ext} = $
\newline

\noindent $ = [ U , U \times _{\pi} \Omega ^n Fred'(H) ]_{\G \rtimes U}^{ext} = $
\newline

\noindent $ = \text{$(\G \rtimes U)$-extendable sections of $ U \times _{\pi} \Omega ^n Fred'(H)$ over $U$} = $
\newline

\noindent $ = \text{$\G $-extendable sections of $ U \times _{\pi} \Omega ^n Fred'(H)$ over $U$}  = $
\newline

\noindent $ = [ U , \Omega ^n Fred'(H) ]_{\G}^{ext} = RK_{\G}^{-n}(U) $
\newline
\newline
 
\noindent $ RK_{G \rtimes M }^*(M) = [ M , M \times Fred'(E) ]_{G \rtimes M}^{ext} = $
\newline

\noindent $ = \text{$(G \rtimes M)$-extendable sections of $ M \times \Omega ^n Fred'(E) $ over $M$} = $
\newline

\noindent $ = \text{$(G \rtimes M)$-sections of $ M \times \Omega ^n Fred'(E) $ over $M$} = $
\newline

\noindent $ = [ M , \Omega ^n Fred'(E) ]_G = K_G^{-n}(M) $
\newline

\noindent Therefore, $ RK_{\G}^{-n}(U) \cong K_G^{-n}(M) $
\end{proof}

\begin{thm}

If $\G$ is a Bredon-compatible finite Lie groupoid, the Bott homomorphism
\[ b = b(X,A) : RK_{\G}^{-n}(X,A) \longrightarrow RK_{\G}^{-n-2}(X,A) \]
is an isomorphism for any finite $\G$-CW-pair $(X,A)$ and all $n \geq 0 $.

\end{thm}

\begin{proof}

Assume first that $ X = Y \cup _{\phi} (U \times D^m) $ where $U \times D^m$ is a $\G$-cell. Assume inductively that $b(Y)$ is
an isomorphism. Since $ RK_{\G}^{-n}(U \times S^{m-1}) \cong RK_G^{-n}( M \times S^{m-1}) $ and 
$ RK_{\G}^{-n}(U \times D^m) \cong RK_G^{-n}( M \times D^m) $, the Bott homomorphisms $b(U \times S^{m-1})$ and 
$b(U \times D^m)$ are isomorphisms by the equivariant Bott periodicity theorem for actions of compact Lie groups. The Bott map
is natural and compatible with the boundary operators in the Mayer-Vietoris sequence for $Y$, $X$, $U \times S^{m-1}$ and
$ U \times D^m $ and so $b(X)$ is an isomorphism by the $5$-lemma. The proof that $b(X,A)$ is an isomorphism follows immediately from the
definitions of the relative groups.
\end{proof}

\noindent Based on the Bott isomorphism we just proved, we can now redefine for all $n \in \Z $
\[ RK_{\G}^n(X,A) = \left \{ \begin{array}{cc}
RK_{\G}^0(X,A) & \text{if $n$ is even} \\
RK_{\G}^{-1}(X,A) & \text{if $n$ is odd} \end{array} \right. \]
For any finite $\G$-CW-pair $(X,A)$, define the boundary operator $\delta ^n : RK_{\G}^n(A) \longrightarrow RK_{\G}^{n+1}(X,A) $ to
be $ \delta : K_{\G}^{-1}(A) \longrightarrow K_{\G}^0(X,A) $ if $n$ is odd, and to be the composite
\[ RK_{\G}^0(A) \stackrel{b}{\longrightarrow} RK_{\G}^{-2}(A) \stackrel{\delta ^{-2}}{\longrightarrow} RK_{\G}^{-1}(X,A) \]
if $n$ is even.
\newline

We can collect all the information we have about $\G$-equivariant representable $K$-theory in the following theorem:

\begin{thm}

If $\G $ is a Bredon-compatible finite Lie groupoid, the groups $RK_{\G }^n(X,A) $ define a $ \Z /2 $-graded multiplicative cohomology 
theory on the category of finite $\G$-CW-pairs.  

\end{thm}

Note that for a general finite Lie groupoid $\G$, $RK_{\G }^*(-) $ is a multiplicative cohomology theory on the category of 
$\G$-spaces, but it is not clear whether we have Bott periodicity.

\begin{cor}

Let $\G$ be a Bredon-compatible finite Lie groupoid and $U$ a $\G$-cell. Then $ K_{\G}^*(U) \cong RK_{\G}^*(U) $

\end{cor}

\begin{proof}

If $U$ is a $\G$-cell, then $\G \rtimes U$ is weakly equivalent to $ G \rtimes M $ for some compact Lie group
$G$ and a finite $G$-CW-complex $M$. In \cite{C} it is proved that $ K_{\G}^*(U) \cong K_G^*(M) $. By corollary
2.20, we also have $ RK_{\G}^*(U) \cong K_G^*(M) $.
\end{proof}

\begin{thm}

Let $\G$ be a Bredon-compatible finite Lie groupoid and $X$ a finite $\G$-CW-complex. Then $ K_{\G}^*(X) \cong RK_{\G}^*(X) $

\end{thm}

\begin{proof}

Assume first that $ X = Y \cup _{\phi} (U \times D^m) $ where $U \times D^m$ is a $\G$-cell. Assume inductively that we have
an isomorphism $ K_{\G}^*(Y) \stackrel{\cong}{\longrightarrow} RK_{\G}^*(Y) $. We know that $ K_{\G}^*(U \times S^{m-1}) \cong RK_{\G}^*(U \times S^{m-1}) $ 
and $ K_{\G}^*(U \times D^m) \cong RK_{\G}^*(U \times D^m) $ by the previous corollary. In fact, since these last two
isomorphism follow from choosing a weak equivalence from the same $\G$-cell to the action of a compact Lie group on a finite
equivariant CW-complex, these isomorphisms are natural with respect to the Mayer-Vietoris sequences for $RK_{\G}^*(-)$ and
$K_{\G}^*(-)$. Let $ RA^{-n} = RK_{\G}^{-n}(Y) \oplus  RK_{\G}^{-n}(U \times D^m) $, $ A^{-n} = K_{\G}^{-n}(Y) \oplus  K_{\G}^{-n}(U \times D^m) $,
$ RB^{-n} = RK_{\G}^{-n}(U \times S^{m-1}) $ and $ B^{-n} = K_{\G}^{-n}(U \times S^{m-1}) $ , then:
\diagram
A^{-n-1} & \rTo & B^{-n-1} & \rTo & K_{\G}^{-n}(X) & \rTo & A^{-n} & \rTo & B^{-n} \\
\dTo^{\cong} & & \dTo^{\cong} & & \dTo & & \dTo^{\cong} & & \dTo^{\cong} \\
RA^{-n-1} & \rTo & RB^{-n-1} & \rTo & RK_{\G}^{-n}(X) & \rTo & RA^{-n} & \rTo & RB^{-n} \\ 
\enddiagram
And so the result follows by the $5$-lemma.
\end{proof}

In other words, we have just proved that the cohomology theory $K_{\G}^*(-)$ is representable by extendable
maps.

\begin{cor}

Let $\G$ be a Bredon-compatible finite Lie groupoid, $X$ a finite $\G$-CW-complex and $H$ a stable representation of $\G$, then:
\[ K_{\G}^n(X) = \left \{ \begin{array}{cc}
[X , Fred'(H)]_{\G}^{ext} & \text{if $n$ is even} \\ $ $
[X, \Omega Fred'(H)]_{\G}^{ext} & \text{if $n$ is odd} \end{array} \right. \]

\end{cor}

\noindent We would like to make two observations:

\begin{enumerate}

\item Note that all constructions and results in this sections remain true if we relax the condition of $\G$ being finite
to $\G$ having a locally universal Hilbert representation $U(\G)$.

\item This cohomology theory should not be confused with the $\G$-equivariant representable $K$-theory defined in \cite{EM}. In 
their paper, they define $\G$-equivariant representable $K$-theory of $X$ as the $KK$-groups associated to $C_0(X)$ and show
that this is actually representable (by all $\G$-equivariant continuous maps) by a corresponding Fredholm bundle. Note 
that in our case only a special class of maps are considered to have a correspondence with extendable vector bundles, but 
the Fredholm bundles in both cases are equivalent. 

\end{enumerate}
\section{Twisted equivariant $K$-theory}

\begin{defn}

Let $\G $ be a Lie groupoid and $X$ a $\G $-space. A $\G$-projective bundle on $X$ is a $\G $-space $P$ with a $\G$-equivariant map
$ p : P \longrightarrow X $ such that there exists an equivariant open covering $ \{ U_i \} $ of $X$ for which $ P_{|U_i} =
U_i \times _{\pi } \P (E) $ for some $\G $-Hilbert bundle $E$ on $G_0 $. Moreover, We shall call $P$ a $\G$-stable projective 
bundle if $ P \cong P \otimes \pi^*\P (U(\G )) $ for some locally universal $\G$-Hilbert bundle $U(\G )$ on $G_0$.

\end{defn}

Let $X$ be a $\G$-space and $P \to X $ a $\G$-projective bundle on $X$. We can then construct the bundle $End(P)$ on $X$
whose fibre at $x$ is the vector space $End(H_x)$ of endomorphisms of a Hilbert space $H_x$ such that $P_x = \P (H_x) $. 
Similarly, we can replace $End(H_x)$ by $Fred(H_x)$, the space of Fredholm operators from $H_x$ to $H_x$, and define in
this way a bundle $ Fred(P) \to X $. Now consider the subbundle $Fred'(P) $ of Fredholm operators $A$ such
that $g \to gAg^{-1}$ is continuous for all $ g \in G_1 $ for which the expression makes sense. 

\begin{defn}

We say that a $\G$-equivariant section $s$ of $Fred'(P) \to X $ is extendable if there is another $\G$-equivariant section 
$t$ such that $ts = v \pi _X $ for some section $v$ of $Fred'(P) \to G_0 $.

\end{defn}

\begin{defn}

We say that a homotopy $ H : X \times I \to Fred'(P) $ of $\G$-equivariant sections is extendable if each $H_t$ is an
extendable $\G$-equivariant section of $ Fred'(P) \to X $.

\end{defn}

\begin{defn}

Let $P$ be a $\G$-stable projective bundle and $ X $ a $\G$-space. We define the $\G$-equivariant twisted $K$-theory of $X$
with twisting $P$ to be the group of extendable homotopy classes of extendable $\G$-equivariant sections of $Fred'(P)$ and 
we denote it by $ ^PK_{\G}(X) $

\end{defn}

In order to define the rest of the twisted $K$-groups, we need to introduce the fibrewise iterated loop-space 
$ \Omega_X^n Fred'(P) $, which is a $\G$-bundle on $X$ whose fibre at $x$ is $ \Omega ^n Fred'(H_x) $.

\begin{defn}

The extendable homotopy classes of sections of this bundle will be denoted by $ ^PK^{-n}_{\G}(X) $.

\end{defn}

The groups $ ^PK^{-n}_{\G}(X) $ are functorially associated to the pair $(X,P)$ and so an isomorphism $ P \to P'$ of 
$\G$-stable projective bundles on $X$ induces an isomorphism $ {^PK^{-n}_{\G}(X)} \to {^{P'}K^{-n}_{\G}(X)} $ for all
$n \geq 0 $

\begin{cor}

If $\G$ is a Bredon-compatible finite Lie groupoid and $P$ is a trivial $\G$-stable projective bundle on a finite
$\G$-CW-complex $X$, then $ {^PK^*_{\G}(X)} \cong K^*_{\G}(X) $.

\end{cor}

\begin{proof}
It follows from the representability of $\G$-equivariant $K$-theory, that is, corollary 2.25.
\end{proof}

\begin{cor}

Let $P$ be a $\G$-stable projective bundle on $Y$. If $f_0$, $f_1 : X \longrightarrow Y $ are $\G$-homotopic $\G$-maps 
between $\G$-spaces, then $f_0^*(P)$ is isomorphic to $f_1^*(P) $ and we have a commutative diagram:

\diagram
{^PK_{\G}^{-n}(Y)} & \rTo^{f_0^*}_{\cong} & {^{f_0^*(P)}K_{\G}^{-n}(X)}  \\
 & \rdTo^{f_1^*}{\cong} & \dTo^{\cong} \\
 &                      &{^{f_1^*(P)}K_{\G}^{-n}(X)} \\
\enddiagram
for all $n \geq 0 $.

\end{cor}

\begin{lemma}

Let
\diagram
A & \rTo^{i_1} & X_1 \\
\dTo_{i_2} & & \dTo_{j_1} \\
X_2 & \rTo^{j_2} & X \\
\enddiagram
be a pushout square of $\G$-spaces and $P$ a $\G$-stable projective bundle on $X$. Let $P_k = j_k^*(P)$ for $k=1,2$ and
$P_A = (i_2)^*(P_2) $. Then there is a natural exact sequence, infinite to the left
\[  \ldots \stackrel{d^{-n-1}}{\longrightarrow} {^PK_{\G}^{-n}(X)} \stackrel{j_1^* \oplus j_2^*}{\longrightarrow} {^{P_1}K_{\G}^{-n}(X_1)} \oplus
    {^{P_2}K_{\G}^{-n}(X_2)} \stackrel{i_1^*-i_2^*}{\longrightarrow} {^{P_A}K_{\G}^{-n}(A)} \stackrel{d^{-n}}{\longrightarrow} \ldots \]
\[  \ldots \longrightarrow {^{P_A}K_{\G}^{-1}(A)} \stackrel{d^{-1}}{\longrightarrow} {^PK_{\G}^0(X)} \stackrel{j_1^* \oplus j_2^*}{\longrightarrow} 
    {^{P_1}K_{\G}^0(X_1)} \oplus {^{P_2}K_{\G}^0(X_2)} \stackrel{i_1^*-i_2^*}{\longrightarrow} {^{P_A}K_{\G}^0(A)} \] 

\end{lemma}

\begin{proof}
The proof is essentially the same as that of lemma 2.16. 
\end{proof}

There is also a multiplication:
\[ {^PK_{\G}^{-n}(X)} \otimes {^{P'}K_{\G}^{-m}(X)} \longrightarrow {^{P\otimes P'}K_{\G}^{-n-m}(X)} \]
coming from the map $(A,A') \to A \otimes 1 + 1 \otimes A' $. This extends the multiplication in untwisted $\G$-equivariant 
$K$-theory and makes ${^PK_{\G}^*(X)}$ into a $K_{\G}^*(X)$-module.
\newline

Just like in the case of representable $K$-theory, for any $\G$-stable Hilbert bundle there is a $\G$-map 
$ \Omega_X^n Fred'(H) \to \Omega_X^{n+2} Fred'(H) $. Therefore, for any $\G$-stable
projective bundle $P$ on $X$ there is a Bott map:
\[ b(X,P) : {^PK^{-n}_{\G}(X)} \longrightarrow {^PK^{-n-2}_{\G}(X)} \]
We do not know if this map is an isomorphism in general. Now we will prove that $b(X,P)$ is an isomorphism when $ X$ is a 
finite $\G$-CW-complex using a similar argument to the one used for untwisted $\G$-equivariant $K$-theory \cite{C}.

\begin{prop}

Suppose that $ F : \G \longrightarrow \H $ is a local equivalence. Then the pullback functor
\[ F^* : \{ \text{Fiber bundles on } \H \} \longrightarrow \{ \text{Fiber bundles on } \G \} \]
is an equivalence of categories.

\end{prop}

\begin{proof}

See \cite{FHT}
\end{proof}

\begin{prop}

Let $ F : \G \longrightarrow \H $ be a local equivalence, and $P$ a $\H$-stable projective bundle on $H_0$. Then we have
an isomorphism $ F^* : {^PK_{\H}^*(H_0)} \longrightarrow {^{F^*(P)}K_{\G}^*(G_0)} $

\end{prop}

\begin{proof}
First of all, $F^*(P)$ is a $\G$-stable projective bundle by proposition 2.8 and the previous proposition. Since these groups 
are defined using sections, the result follows from proposition 2.17. 
\end{proof}

\begin{cor}
If $\G$ and $\H$ are weakly equivalent and $P$ a $\H$-stable projective bundle on $H_0$, then $ {^PK_{\H}^*(H_0)} \longrightarrow 
{^{F^*(P)}K_{\G}^*(G_0)} $
\end{cor}

\begin{cor}

If $\G$ is a Bredon-compatible finite Lie groupoid, $U$ is a $\G$-cell and $P$ is a $\G$-stable projective bundle on $U$, then 
${^PK_{\G}^*(U)} \cong {^QK_G^*(M)} $ for some compact Lie group $G$, some finite $G$-CW-complex $M$ and some $G$-stable 
projective bundle $Q$ on $M$.

\end{cor}

\begin{proof}
Since $U$ is a $G$-cell, we know that $ \G \rtimes U $ is weakly equivalent to $ G \rtimes M $ for some compact Lie group
$G$ and a finite $G$-CW-complex $M$. Therefore, by the previous corollary:
\[ {^PK_{\G \rtimes U}^*(U)} \cong {^QK_{G \rtimes M}^*(M)} \]
for some $G \rtimes M$-stable projective bundle $Q$ on $M$. In the proof of corollary 2.20, we saw that if $H$ is a locally 
universal representation of $\G$, then $ U \times _{\pi} H $ is a locally universal $\G \rtimes U$-Hilbert bundle. And also 
that if $E$ is a locally universal Hilbert representation of $G$, then $ E \times M $ is a locally universal 
$G \rtimes M $-Hilbert bundle. It follows that if $P$ is a $\G \rtimes U$-stable projective bundle on $U$, then $P$ is a 
$\G$-stable projective bundle on $U$. Similarly, if $Q$ is a $ G \rtimes M$-stable projective bundle on $M$, then $Q$ is a 
$G$-stable projective bundle on $M$.
\newline

\noindent $ {^PK_{\G \rtimes U}^*(U)} = $
\newline

\noindent $ = \text{$(\G \rtimes U)$-extendable sections of $ \Omega ^n Fred'(P)$ over $U$} = $
\newline

\noindent $ = \text{$\G $-extendable sections of $ \Omega ^n Fred'(P)$ over $U$}  = $
\newline

\noindent $ = {^PK_{\G}^*(U)} $
\newline
\newline
 
\noindent $ {^QK_{G \rtimes M}^*(M)} = $
\newline

\noindent $ = \text{$(G \rtimes M)$-extendable sections of $ \Omega ^n Fred'(Q) $ over $M$} = $
\newline

\noindent $ = \text{$G$-sections of $ \Omega ^n Fred'(Q) $ over $M$} = $
\newline

\noindent $ = {^QK_G^*(M)} $
\newline

\noindent Therefore, $ {^PK_{\G}^*(U)} \cong {^QK_G^*(M)} $

\end{proof}

\begin{thm}

If $\G$ is a Bredon-compatible finite Lie groupoid, the Bott homomorphism
\[ b = b(X,P) : {^PK_{\G}^{-n}(X)} \longrightarrow {^PK_{\G}^{-n-2}(X)} \]
is an isomorphism for any finite $\G$-CW-complex $X$, all $\G$-stable projective bundles on $X$ and all $n \geq 0 $.

\end{thm}

\begin{proof}

Assume that $ X = Y \cup _{\phi} (U \times D^m) $ where $U \times D^m$ is a $\G$-cell. Let $P$ be a $\G$-stable
projective bundle. Assume inductively that $b(Y,P|Y)$ is
an isomorphism. Since $ ^{P|(U \times D^m)}K_{\G}^{-n}(U \times D^m) \cong {^QK_G^{-n}( M \times D^m)} $ and 
$ {^{P|(U \times S^{m-1})}K_{\G}^{-n}(U \times S^{m-1})} \cong {^{Q|(M \times S^{m-1})}K_G^{-n}( M \times S^{m-1})} $, 
the Bott homomorphisms $b(U \times S^{m-1},P_{|(U \times S^{m-1})})$ and 
$b(U \times D^m,P_{|(U \times D^m)})$ are isomorphisms by the Bott periodicity theorem in twisted equivariant $K$-theory for 
actions of compact Lie groups \cite{AS2}. The Bott map is natural and compatible with the boundary operators in the 
Mayer-Vietoris sequence for $Y$, $X$, $U \times S^{m-1}$ and $ U \times D^m $ and so $b(X,P)$ is an isomorphism by the 
$5$-lemma.
\end{proof}

\noindent Based on the Bott isomorphism we just proved, we can now redefine for all $n \in \Z $
\[ ^PK_{\G}^n(X) = \left \{ \begin{array}{cc}
^PK_{\G}^0(X) & \text{if $n$ is even} \\
^PK_{\G}^{-1}(X) & \text{if $n$ is odd} \end{array} \right. \]
We can collect all the information we have so far about $\G$-equivariant $K$-theory in the following theorem:

\begin{thm}

If $\G $ is a Bredon-compatible finite Lie groupoid, the groups ${^PK_{\G }^n(X)} $ define a $ \Z /2 $-graded cohomology 
theory on the category of finite $\G$-CW-complexes with $\G$-stable projective bundles, which is a module over untwisted
$\G$-equivariant $K$-theory.  

\end{thm}

Note that for a general Lie groupoid $\G$, $K_{\G }^*(-) $ is a cohomology theory on the category of 
$\G$-spaces, but it is not clear whether we have Bott periodicity.
\newline

All constructions and results in this section are true if we relax the condition of $\G$ finite to $\G$ admitting a locally universal
$\G$-representation.

\section{The completion theorem}

\noindent For any $\G$-stable projective bundle $P$ on a $\G$-space $X$, consider the $\G$-stable projective bundle 
$ P \times _{\pi } E^n\G $ on $ X \times _{\pi} E^n\G $. The following diagram commutes:
\newline
\diagram
P \times _{\pi } E^n\G & \rTo & P \\
\dTo & & \dTo \\
X \times _{\pi} E^n\G & \rTo & X \\
\enddiagram
Therefore we have a map:
\[ ^PK_{\G }^* (X) \longrightarrow {^{P \times _{\pi } E^n\G}K_{\G }^* (X \times _{\pi } E^n\G) } \]
$^PK_{\G }^* (X)$ is a module over $K_{\G}^*(G_0)$ and ${^{P \times _{\pi } E^n\G}K_{\G }^* (X \times _{\pi } E^n\G) }$ is a module
over $ K_{\G}^*(E^n\G) $. In fact we have a commutative diagram:
\diagram
K_{\G}^*(G_0) & \rTo & K_{\G}^*(E^n\G) \\
\dTo & & \dTo \\
^PK_{\G }^* (X) & \rTo & {^{P \times _{\pi } E^n\G}K_{\G }^* (X \times _{\pi } E^n\G) } \\
\enddiagram
\noindent From \cite{C}, we know that the last map factors through $I_{\G}^n $ and therefore, by naturality we have a map:
\[ ^PK_{\G }^* (X)/I_{\G}^n{^PK_{\G }^* (X)} \longrightarrow {^{P \times _{\pi } E^n\G}K_{\G }^* (X \times _{\pi } E^n\G) } \]
\noindent We can also look at these maps as a map of pro-$K_{\G}^*(G_0)$-modules:
\[ \{ ^PK_{\G }^* (X)/I_{\G}^n{^PK_{\G }^* (X)} \} \longrightarrow \{ {^{P \times _{\pi } E^n\G}K_{\G }^* (X \times _{\pi } E^n\G) } \} \]
\noindent Taking limits we obtain a map of $K_{\G}^*(G_0)$-modules:
\[  ^PK_{\G }^* (X)_{I_{G}}^{\wedge} \longrightarrow  {^{P \times _{\pi } E\G}K_{\G }^* (X \times _{\pi } E\G) } \]

\begin{conj}

Let $\G$ be a finite Lie groupoid, $X$ a $\G$-space and $P$ a $\G$-stable projective bundle. Then we have an isomorphism of $K_{\G}^*(G_0)$-modules: 
\[  ^PK_{\G }^* (X)_{I_{G}}^{\wedge} \longrightarrow  {^{P \times _{\pi } E\G}K_{\G }^* (X \times _{\pi } E\G) } \]
\end{conj}

If a groupoid $\G$ satisfies this conjecture for $ X = G_0 $ and all $\G$-stable projective bundles on $G_0$, we will say $\G$ 
satisfies the twisted completion theorem.

\begin{lemma}

Let $\G = G \rtimes X $, where $G$ is a compact Lie group and $X$ is a compact $G$-space such that $K_G^*(X)$ is finite
over $R(G)$. Then $\G$ satisfies the twisted completion theorem. 

\end{lemma}

\begin{proof}

It follows from lemma 5.1 in \cite{C} and the completion theorem for twisted equivariant $K$-theory for actions of compact
Lie groups \cite{La}
\end{proof}

\begin{lemma}

If $\G$ and $\H$ are locally equivalent, then $\G$ satisfies the twisted completion theorem if and only if $\H$ does.

\end{lemma}

\begin{proof} Let $P$ be a $\H$-stable projective bundle on $H_0$. Then, $F^*(P)$ is $\G$-stable. If $\G$ and 
$\H$ are locally equivalent by a local equivalence $ F : \H \longrightarrow \G $, then we have an isomorphism 
$ f : {^{F^*P}K_{\G}^*(G_0)} \stackrel{\cong}{\longrightarrow} {^PK_{\H}^*(H_0)} $. The following diagram is commutative: 
\diagram
K_{\G}^*(G_0) & \rTo^{\cong} & K_{\H}^*(H_0) \\
\dTo & & \dTo \\
{^{F^*P}K_{\G}^*(G_0)} & \rTo^{\cong} & {^PK_{\H}^*(H_0)} \\
\enddiagram
By lemma 5.2 in \cite{C}, the topologies induced by $I_{\G}$ and $I_{\H} $ are the same and therefore we have an 
isomorphism of pro-rings 
$ \{ {^{F^*P}K_{\G}^*(G_0)}/I_{\G}^n{^{F^*P}K_{\G}^*(G_0)} \} \cong \{ {^PK_{\H}^*(H_0)}/I_{\H}^n{^PK_{\H}^*(H_0)} \} $.
\newline

The local equivalence also induces a local equivalence between the groupoids $\G \rtimes E\G$ and $\H \rtimes E\H$ and 
it takes $ E^n\H$ to $ E^n\G $. Hence we have an homomorphism of pro-rings 
$ \{ {^{F^*(P \times _{\pi} E^n\H)}K_{\G}^*(E^n\G)} \} \cong \{ {^{P \times _{\pi} E^n\H}K_{\H}^*(E^n\H)} \} $, which
is an isomorphism in the limit. We also have $ F^*(P \times _{\pi} E^n\H) = F^*(P) \times _{\pi} E^n\G $. We have a commutative
diagram:
\diagram
\{ {^{F^*P}K_{\G}^*(G_0)}/I_{\G}^n{^{F^*P}K_{\G}^*(G_0)} \} & \rTo^{\cong} & \{ {^PK_{\H}^*(H_0)}/I_{\H}^n{^PK_{\H}^*(H_0)} \} \\
\dTo & & \dTo \\
\{ {^{F^*(P) \times _{\pi} E^n\G}K_{\G}^*(E^n\G)} \} & \rTo & \{ {^{P \times _{\pi} E^n\H}K_{\H}^*(E^n\H)} \} 
\enddiagram
The lemma follows then by looking at the diagram:
\diagram
{^{F^*P}K_{\G}^*(G_0)_{I_{\G}}^{\wedge}} & \rTo^{\cong} & {^PK_{\H}^*(H_0)_{I_{\H}}^{\wedge}} \\
\dTo & & \dTo \\
{^{F^*(P) \times _{\pi} E\G}K_{\G}^*(E\G)} & \rTo^{\cong} & {^{P \times _{\pi} E\H}K_{\H}^*(E\H)}   \qedhere \\ 
\enddiagram  
\end{proof}

\noindent From the previous lemma, we obtain the following theorem:

\begin{thm}

If $\G$ and $\H$ are weakly equivalent, then $\G$ satisfies the twisted completion theorem if and only if $\H$ does.

\end{thm}

\noindent Now from this theorem, lemma 4.2 and the results in \cite{La}, we obtain this corollary:

\begin{cor}

If $\G$ is a Bredon-compatible finite Lie groupoid, $U$ is a $\G$-cell and $P$ is a stable $\G$-projective bundle on $U$, $^PK_{\G}^*(U) $ is a 
finitely generated abelian group and the groupoid $\G \rtimes U $ satisfies the completion theorem.

\end{cor}

This corollary tells us that the twisted completion theorem is true for $\G$-cells. Now we move on to prove this for finite
$\G$-CW-complexes.

Let $X$ be a finite $\G$-CW-complex and $P$ a $\G$-stable projective bundle on $X$. Consider the spectral sequence for the 
maps $ f : X \rightarrow X / \G $ in twisted $\G $-equivariant $K$-theory with twisting
given by the restrictions of $P$:
\[  E_1 ^{pq} = \mathop{\mathop{\prod}}   _{i \in I_p} {^{f^*(P_i)}K_{\G }^q (f^{-1}U_i)} \Longrightarrow {^PK_{\G }^{p+q} (X)} \]
The spectral sequence $E$ is a spectral sequence of $ K_{\G}^*(G_0) $-modules. Assume $\G$ is a finite groupoid \cite{C} so that
$ K_{\G}^*(G_0) $ is a Noetherian ring. All elements in these spectral sequences are finitely generated. The functor taking a 
$ K_{\G}^*(G_0)$-module $M$ to the $ K_{\G}^*(G_0)$-module $ M_{I_{\G}}^{\wedge} $ is exact \cite{LO} and so we can 
form the following spectral sequence of $ K_{\G}^*(G_0)$-modules. 
\[  F_1 ^{pq} = \mathop{\mathop{\prod}}   _{i \in I_p} {^{Q_i}K_{\G }^q (f^{-1}U_i)_{I_{\G}}^{\wedge} }
\Longrightarrow  {^PK_{\G }^{p+q} (X)_{I_{\G}}^{\wedge}}  \]
where $Q_i = f^*(P_i) $. Similarly consider the map $ h : X \times _{\pi} E\G / \G \longrightarrow X/\G $. It gives us another 
spectral sequence of $ K_{\G}^*(G_0)$-modules:
\[  \bar{F}_1 ^{pq} =  \mathop{\mathop{\prod}}   _{i \in I_p} {^{Q_i \times _{\pi} E\G}K_{\G }^q (h^{-1}U_i)}
\Longrightarrow  {^{P \times _{\pi} E\G}K_{\G }^{p+q} (X \times _{\pi} E\G)} \]
since $ h^*(P_i) = Q_i \times _{\pi } E\G $. We have $h^{-1}(U) = f^{-1}(U) \times _{\pi } E\G $ 
so there is a map of spectral sequences $ F \rightarrow \bar{F} $ induced by the projections $ f^{-1}(U) \times _{\pi } E\G 
\rightarrow f^{-1}(U) $. Let us denote 

If $\G$ is Bredon-compatible, the groupoids $ \G \rtimes f^{-1}(U_i) $ satisfy the twisted completion theorem for all $i$. From
\cite{C}, we know that the topologies determined by the groupoid $\G$ and $\G \rtimes U_i $ on $K_{\G}^*(f^{-1}U_i) $ are the same
and therefore they are the same on $^{Q_i}K_{\G}^*(f^{-1}(U_i)) $.

This proves $\phi$ is an isomorphism when restricted to any particular element $F^{ij}$ and therefore, it is an isomorphism
of spectral sequences. In particular, we have $ {^PK_{\G }^{p+q} (X)_{I_{\G}}^{\wedge}} \cong  
{^{P \times _{\pi} E\G}K_{\G }^{p+q} (X \times _{\pi} E\G)} $

\begin{thm}
Let $\G$ be a Bredon-compatible finite Lie groupoid, $X$ a finite $\G$-CW-complex and $P$ a $\G$-stable projective bundle on $X$. 
Then we have an isomorphism of $K_{\G}^*(G_0)$-modules: 
\[ {^PK_{\G }^n (X)_{I_{\G}}^{\wedge}} \longrightarrow {^{P \times _{\pi} E\G}K_{\G }^n (X \times _{\pi} E\G)} \]
\end{thm}

\section{Proper actions}

\noindent Throughout this whole section $S$ will be a Lie group, but not necessarily compact. To study proper actions of $S$, we
can consider the groupoid $\G = S \rtimes \underline{E}S$, where $\underline{E}S$ is the universal space for proper actions of 
$S$ as defined in \cite{Lu}. This space is a proper $S$-CW-complex such that $\underline{E}S^{G} $ is contractible for all 
compact Lie subgroups $G$ of $S$. The existence of $\underline{E}S$ is shown in \cite{Lu}. It is also shown there that every 
proper $S$-CW-complex has an $S$-map to $\underline{E}S$ and this map is unique up to $S$-homotopy. Some immediate consequences 
follow:

\begin{itemize}

\item Proper $S$-CW-complexes are $\G$-CW-complexes.

\item $\G $ is finite if and only if $\underline{E}S $ is a finite proper $S$-CW-complex.

\item Extendable $\G$-sections on a proper $S$-CW-complex $X$ are extendable $S$-sections for any $S$-map 
$ X \longrightarrow \underline{E}S $, since all of them are $S$-homotopic.

\item If $H$ is a locally universal $S$-Hilbert representation, then $ \underline{E}S \times H $ is a locally
universal $\G$-Hilbert bundle.

\item Stable $\G$-projective bundles on $X$ are stable $S$-projective bundles on $X$.

\end{itemize}

By abuse of language, we say that proper actions of $S$ are Bredon-compatible if the corresponding groupoid $\G = S \rtimes \underline{E}S$
is Bredon-compatible.

\begin{example}

Twisted $K$-theory for actions of finite groups and compact Lie groups was defined and studied in \cite{AS2}. It is
a well-known fact that for these actions, $\G $ is Bredon-compatible \cite{SG3}. A completion theorem was recently proven in
\cite{La}.

\end{example}

\begin{example}

Twisted $K$-theory for actions of discrete groups for some particular twistings was defined in \cite{D}. These actions are 
also Bredon-compatible, as shown in \cite{LO}. Therefore the constructions in this paper provide a model for twisted $K$-theory
for actions of discrete groups for any possible twisting, and a new completion theorem.

\end{example}

\begin{example}

In general, vector bundles may not be enough to construct an interesting equivariant cohomology theory for proper actions of
second countable locally compact groups \cite{Ph}, but they suffice for two important families, almost compact groups and matrix
groups \cite{Ph2}. Actions of these two families of groups are Bredon-compatible \cite{Ph2}. Using the associated groupoids,
we now have twisted $K$-theory for actions of these groups and a completion theorem.

\end{example}

\begin{example}

Proper actions of totally disconnected groups that are projective limits of discrete groups are shown to be Bredon-compatible
in \cite{S}. The results in this paper show a way of defining twisted $K$-theory and a corresponding completion theorem. 

\end{example}

\begin{example}

In general $\G$ need not be a Bredon-compatible groupoid. When $ \G $ is a Bredon-compatible groupoid we must have
$Vect_{\G}(S/G) = Vect_G(pt) $ for a compact subgroup $G$ of $S$. Let $S$ be a Kac-Moody group and $T$ its maximal torus. Note that 
$S$ is not a Lie group, but all constructions generalize.

There is an $S$-map $ S/T \longrightarrow \underline{E}S $ which is unique, up to homotopy. Given an $S$-vector bundle $V$ on 
$\underline{E}S$, the pullback to $S/T$ is given by a finite-dimensional representation of $T$ invariant under the Weyl group.
This representation gives rise to a finite-dimensional representation of $S$. But all finite-dimensional representations of
Kac-Moody groups are trivial, therefore $V$ must be trivial. In particular, extendable $S$-vector bundles on $S/T$ only
come from trivial representations of $T$. 

In order to deal with these groups, it is more convenient to use dominant $K$-theory, which was developed in \cite{K}. Kac-Moody
groups possess an important class of representations called dominant representations. A dominant representation of a Kac-Moody 
group in a Hilbert space is one that decomposes into a sum of highest weight representations. Equivariant 
$K$-theory for proper actions of Kac-Moody groups is defined as the representable equivariant cohomology theory modeled on
the space of Fredholm operators on a Hilbert space which is a maximal dominant representation of the group. It is expected
that twisted dominant $K$-theory can be defined in the same way using a corresponding Fredholm bundle over a projective bundle
which is stable with respect to a suitable Hilbert space of dominant representations.

\end{example}

\section{The category of $\G$-orbits}

Recall from \cite{C} that a $\G$-orbit is a $\G$-space $U$ for which the groupoid $\G \rtimes U$ is weakly equivalent to the 
action of a compact Lie group $G$ on a finite $G$-CW-complex.

The category of $\G$-orbits, written $O_{\G}$ is a topological category with discrete object space formed by the $\G$-orbits. 
The morphisms are the $\G$-maps, with a topology such that the evaluation maps $ Hom_{\G}(U,V) \times U \to V $ are continuous
for all $\G$-orbits $U$, $V$. By an $O_{\G}$-space we shall mean a continuous contravariant functor from $O_{\G}$ to the category
of topological spaces.

\begin{defn}

Let $X$ be a $\G$-space. The fixed point set system of $X$, written $\Phi X $, is an $O_{\G}$-space defined by
$ \Phi X(U) = Map_{\G}(U,X) $ and given $ \Theta : U \to V $, $\Phi X (\Theta)(f) = f \Theta $. We also denote 
$ X^U = Map_{\G}(U,X) $. 

\end{defn}

\begin{defn}

A CW-$O_{\G}$-space is an $O_{\G}$-space $T$ such that each space $T(U)$ is a CW-complex and each structure map
$ T(U) \to T(V) $ is cellular. We will call $T$ regular if it is homotopy equivalent (in the sense detailed below)
to a CW-$O_{\G}$-space.

\end{defn}

\begin{thm}

There is a functor $C : O_{\G}\text{-spaces} \longrightarrow \G \text{-spaces} $ and a natural transformation 
$ \eta : \Phi C \to Id $ such that for each $O_{\G}$-space $T$ and each $U$, $\eta : (CT)^U \rightarrow T(U) $ is a homotopy 
equivalence (it is actually a strong deformation retraction). If $T$ is regular, then $T$ has the $\G$-homotopy type of a 
$\G$-CW-complex.

\end{thm}

\begin{proof}

We first construct the $\G$-space $CT$. Let $O_T$ denote the topological category whose objects are triples $(U,s,y)$
where $U$ is a $\G$-orbit, $ s \in J(U) \equiv U $ and $ y \in T(U) $. Let us consider the nerve of this category as a 
topological simplicial space. This is the bar complex $B_*(T,O_{\G},J)$, where $J:O_{\G} \to Top $ is the covariant functor
which forgets the $\G$-action.

Then $ B_n(T,O_{\G},J)$ consist of $(n+2)$-tuples $(y,f_1,f_2,\ldots,f_n,s) $ where the $f_i: U_i \to U_{i-1}$ are 
composable arrows in $O_{\G}$, $ s \in J(U_n) \equiv U_n $ and $ y \in T(U_0) $. The boundary maps are given by:
\[ \partial _0 (y,f_1,f_2,\ldots ,f_n,s) = (f_1^*(y),f_2,f_3,\ldots ,f_n,s) \]
\[ \partial _n (y,f_1,f_2,\ldots ,f_n,s) = (y,f_1,f_2,\ldots ,f_{n-1},(f_n)_*(s)) \]
\[ \partial _i (y,f_1,f_2,\ldots ,f_n,s) = (y,f_1,f_2,\ldots , f_{i-1},f_if_{i+1},f_{i+2}, \ldots ,f_n,s) \]
Degeneracies are the insertion of identity maps in the appropriate spots. The groupoid $\G$ acts simplicially on 
$B_*(T,O_{\G},J)$ and consequently the geometric realization $B(T,O_{\G},J)$ is a $\G$-space. We define $CT = B(T,O_{\G},J) $.

We now require the homotopy equivalence $\eta : (CT)^U \rightarrow T(U) $ for each $\G$-orbit $U$, natural in $U$. We have: 
\[ (CT)^U = B(T,O_{\G},J)^U = B(T,O_{\G},Hom_{\G}(U,-)) \]
The second equality follows from the fact that $\G$ acts on the last coordinate only. Now it is a general property of the 
bar construction that for any topological category $C$, contravariant functor $ F : C \to Top $ and object $ A $ of $ C $, there
is a natural map
\[ \eta : B(F,C,Hom_C(A,-)) \longrightarrow F(A) \]
which is a strong deformation retraction. This map is induced by a simplicial map
\[ \eta _* : B_*(F,C,Hom_C(A,-)) \longrightarrow F(A)_* \]
where $F(A)_*$ is the simplicial space all of whose components are $F(A)$ and all whose face and degeneracy maps are the
identity. In our case, $\eta _*$ is given by the formula:
\[ \eta _n(y,f_1,f_2,\ldots ,f_n,f) = (f_1 \circ \ldots \circ f_n \circ f)^*(y) \]
Now $f$ is an element of $Hom_{O_{\G}}(U,U_n) $. The proof that $\eta$ is a strong deformation retraction is a standard
simplicial argument contained in \cite{M}.
\end{proof}

\section{Twistings for groupoid actions}

Now we follow \cite{AS2} closely. Let us write Pic$_{\G}(X) $ for the group of isomorphism classes of complex $\G $-line 
bundles on $X$ (or equivalently, of principal $S^1$-bundles on $X$ with $\G $-action), and Proj$_{\G }(X)$ for the group of 
isomorphism classes of $\G $-stable projective bundles. Applying the Borel construction to line bundles and projective bundles 
gives us homomorphisms: 
\[ Pic_{\G }(X) \longrightarrow Pic(X_{\G}) \cong H_{\G }^2(X;\Z ) \]
\[ Proj_{\G }(X) \longrightarrow Proj(X_{\G }) \cong H_{\G }^3(X;\Z ) \]
which we shall show are bijective.

\begin{defn}

A topological abelian $\G $-module is a $\G $-space such that each of the fibres of its anchor map is a topological abelian
group and the action is linear.

\end{defn}

\begin{example}

Given any topological abelian group $B$, $G_0 \times B $ is a $\G $-module with the anchor map given by projection on the first
coordinate. When there is no danger of confusion we will denote it by $B$.

\end{example}

Let us introduce groups $ H_{\G}^*(X;A) $ defined for any abelian $\G $-module $A$. These are the hypercohomology groups of a
simplicial space $\chi $ whose realization is the space $ X_{\G } $. Whenever a Lie groupoid $\G $ acts on a space $X$ we can
define the action groupoid whose space of objects is $X$ and space of morphisms is $G_1 \times _{\pi } X $. Let $\chi $ be the 
nerve of this groupoid regarded as a simplicial space, that is, $ \chi _p = G_p \times _{\pi} X $, where $G_p$ is the space
of composable $p$-tuples of arrows in $\G$.

For any simplicial space with an action of a Lie groupoid $\G $ and any topological abelian $\G $-module $A$ we can define the 
hypercohomology $ \Hy ^*(\chi ; sh(A) ) $ with coefficients in the sheaf of continuous equivariant $A$-valued functions. It is the
cohomology of a double complex $C^{..} $, where, for each $ p \geq 0 $, the cochain complex $ C^{p.} $ calculates $ H^*(\chi _p; sh(A)) $.

\begin{defn}

$ H_{\G }^*(X;A) = \Hy ^*(\chi ; sh(A)) $

\end{defn}

These groups are the abutment of a spectral sequence with $ E_1^{pq} = H^q(G_p \times _{\pi} X ; sh(A)) $.

\begin{lemma}

If $\G$ is a Bredon-compatible Lie groupoid and $X$ is a finite $\G$-CW-complex, 
$ H_{\G}^{p+1}(X ; \Z ) \cong H_{\G}^p(X ; S^1 ) $ for any $ p > 0 $.

\end{lemma}

\begin{proof}
If we compare the spectral sequences for the $\G$-CW-structure of $X$ with respect to the cohomology theories 
$ H_{\G}^{p+1}(- ; \Z ) $ and $ H_{\G}^p(- ; S^1 ) $, we notice that we have an isomorphism in each cell by the
similar result in \cite{AS2}. 
%Because of the exact sequence
%
%\[ 0 \rightarrow sh(G_0 \times \Z ) \rightarrow sh(G_0 \times \R ) \rightarrow sh(G_0 \times S^1) \rightarrow 0 \]
%
%it is enought to show that $ H_{\G }^p(\chi ;\R ) = 0 $ for $ p > 0 $. As $ E_1^{pq} =0 $ for $ q > 0 $ in the spectral 
%sequence when $ A = G_0 \times \R $, we see that $ H_{\G }^*(X;  \R ) $ is simply the cohomology of the cochain complex of 
%continuous equivariant real-valued functions on the simplicial space $\chi $, which is easily recognized as the complex of 
%continuous Eilenberg-Maclane cochains of the groupoid $\G$ with values in the topological vector space 
%Map$(X,G_0 \times \R )$. This complex is acyclic in positive degrees. It is the $\G $-invariant part of the contractible 
%complex of so-called homogeneous cochains, and taking the $\G $-invariants is an exact functor, simply because cochains can be 
%averaged over $\G $.
\end{proof}

\begin{prop}

Let $\G$ be a Bredon-compatible finite Lie groupoid and $X$ a $\G $-space. Then we have:

\begin{enumerate}

\item $ H_{\G }^2(G_0; \Z ) \cong Hom(\G , G_0 \times S^1 ) $

\item $ H_{\G }^3(G_0; \Z ) \cong Ext(\G , G_0 \times S^1 ) $, the group of central extensions $ 1 \rightarrow
G_0 \times S^1 \rightarrow \H \rightarrow \G \rightarrow 1 $.

\item $ H_{\G }^2(X; \Z ) \cong Pic_{\G }(X) $

\item $ H_{\G }^3(X; \Z ) \cong Proj_{\G }(X) $

\end{enumerate}

\end{prop}

\begin{proof}

\begin{enumerate} 

\item When $ X = G_0 $, we have $ E_1^{0q} = H^q(G_0;sh(G_0 \times S^1)) \cong H^q(pt;sh(S^1)) = 0 $ and in the previous lemma
we have seen that $ E_2^{p0} = H_{c.c.}^p(\G ; S_1 ) $ is the cohomology of $\G $ defined by continuous Eilenberg-Maclane
cochains. So
\[  H_{\G }^2(G_0 ; \Z ) = H_{\G }^1(G_0 ; S^1 ) \cong E_2^{10} \cong H_{c.c.}^1(\G ; S_1 )  
\cong Hom(\G , G_0 \times S^1 ) \]
\item In this case the spectral sequence gives us an exact sequence
\[ 0 \rightarrow E_2^{20} \rightarrow H_{\G }^2(G_0 ; S^1 ) \rightarrow E_2^{11} \rightarrow E_2^{30} \]
that is,
\[ 0 \rightarrow H_{c.c.}^2(\G ; S_1 ) \rightarrow H_{\G }^2(G_0 ; S^1 ) \rightarrow Pic(\G)_{prim} \rightarrow H_{c.c.}^3(\G ; S_1 ) \]
for $ E_1^{11} = H^1(\G;sh(G_0 \times S^1)) = Pic(\G ) $, and $ E_2^{11} $ is the subgroup of primitive elements, that is, of circle
bundles $ \H $ on $ \G $ such that $ m^*\H \cong pr_1^\H \otimes pr_2^*\H $, where $ pr_1,pr_2,m: G_1 \times _{G_0} G_1 \longrightarrow G_1 $
are the obvious maps. Equivalently, $Pic(G)_{prim} $ consists of circle bundles $\H $ on $\G $ equipped with bundle maps $ \tilde{m} 
: H_1 \times _{G_0 } H_1 \longrightarrow H_1 $ covering the multiplication in $\G $. It is easy to see that the composite
\[ Ext(\G , G_0 \times S^1 ) \rightarrow H_{\G }^2(G_0 ; G_0 \times S^1 ) \rightarrow Pic(\G ) \]
takes an extension to its class as a circle bundle. On the other hand $ H_{c.c.}^2(\G ; S^1 ) $ is plainly the group
of extensions $ G_0 \times S^1 \rightarrow \H \rightarrow \G $ which as circle bundles admit a continuous section, so its image
in Ext$(\G , G_0 \times S^1 ) $ is precisely the kernel of this composite. It remains only to show that the image of Ext$(\G , G_0 \times S^1)$
in Pic$(\G)_{prim}$ is the kernel of $ Pic(\G)_{prim} \longrightarrow H_{c.c.}^3(\G ; S_1 ) $. This map associates to
a bundle $ \H $ with a bundle map $ \tilde{m} $ as above precisely the obstruction to changing $\tilde{m} $ by a bundle map
$ G_1 \times _{G_0 } G_1 \longrightarrow S^1 $ to make it an associative product on $ \H $.
\newline

\item The spectral sequence gives
\[ 0 \rightarrow E_2^{10} \rightarrow H_{\G }^1(X; S^1) \rightarrow E_2^{01} \rightarrow E_2^{20} \]
Now $ E_1^{01} = Pic(X) $, and $ E_2^{01} $ is the subgroup of circle bundles $ S \longrightarrow X $ which admit a bundle
map $ \tilde{m} : G_1 \times _{G_0 } S \longrightarrow S $ covering the $ \G $-action on $X$. As before, $\tilde{m} $ can be
made into a $\G $-action on $S$ if and only if an obstruction in $ H_{c.c.}^2(\G ; Map(X,G_0 \times S^1)) $ vanishes. Finally,
the kernel of $ Pic_{\G }(X) \longrightarrow Pic(X) $ is the group of $\G $-actions on $ X \times S^1 $, and this is just
$ E_2^{10} = H_{c.c.}^1(\G ; Map(X, G_0 \times S^1)) $.
\newline

\item First we shall prove that the map $ Proj_{\G }(X) \longrightarrow H_{\G }^3(X; \Z) $ is injective.

Consider the filtration
\[ Proj_{\G }(X) \supseteq Proj^{(1)} \supseteq Proj^{(0)} \]
Here $ Proj^{(1)} $ consists of the stable projective bundles which are trivial when the $\G $-action is forgotten, that is, those
that can be described by cocycles $ \alpha : G_1 \times _{G_0 } X \longrightarrow PU(\H) $ such that
$ \alpha (g_2,g_1x) \alpha (g_1,x) = \alpha (g_2g_1,x) $.

$Proj^{(0)} $ consists of those projective bundles for which $\alpha $ lifts to $ \bar{\alpha} : G_1 \times _{G_0 } X \longrightarrow U(\H ) $ satisfying
$ \alpha (g_2,g_1x) \alpha (g_1,x) = c(g_2,g_1,x) \alpha (g_2g_1,x) $ for some $ c : G_2 \times _{\pi } X \longrightarrow S^1 $

We shall compare the filtration of Proj$_{\G }(X) $ with the filtration
\[ H_{\G }^2(X; S^1) \supset H^{(1)} \supset H^{(0)} \]
defined by the spectral sequence. By definition $ H^{(1)} $ is the kernel of $ H_{\G}^2(X ; S_1 ) \longrightarrow
E_1^{02} = H^2(X ; sh( S_1 )) = Proj(X) $, and the composite $ Proj_{\G } \rightarrow H_{\G }^2(X; S_1) \rightarrow
Proj(X) $ is clearly the map which forgets the $\G $-action. Thus $Proj_{\G }(X) / Proj^{(1)} $ maps injectively to $ H_{\G}^2(X,S^1) /
H^{(1)} $

Now let us consider the map $ Proj^{(1)} \longrightarrow H^{(1)} $. The subgroup $ H^{(0)} $ is the kernel of $ H^{(1)} \longrightarrow
E_2^{11} $, while $ E_1^{11} = Pic (G_1 \times _{G_0} X) $. We readily check that an element of $ Proj^{(1)}$ defined by the 
cocycle $\alpha $ maps to the element of $Pic(G_1 \times _{G_0 } X) $ which is the pullback of the circle bundle $U(\H ) \longrightarrow
PU(\H ) $, and can conclude that $\alpha $ maps to zero in $ E_2^{11} $ if and only if it defines an element of $Proj^{(0)} $.
Thus $ Proj^{(1)} / Proj^{(0)} $ injects into $ H^{(1)} / H^{(0)} $. Finally, assigning to an element $\alpha $ of $Proj^{(0)} $
the class in $ E_2^{20} = H_{c.c.}^2(\G ; Map (X, G_0 \times S^1 ) ) $ of the cocycle $c$, we see that if this class vanishes,
then the projective bundle comes from a $\G $-Hilbert bundle, which is necessarily trivial, as we have already explained. So
$Proj^{(0)} $ injects into $ H^{(0)} $.

Now we will construct a universal $\G $-space $C(P)$ with a natural $\G $-stable projective bundle on it and show that the composite 
map
\[  [ X , C(P) ] _{\G} \rightarrow Proj_{\G }(X) \rightarrow H_{\G}^3(X;\Z ) \]
is an isomorphism.

Let $U$ be a $\G$-CW-cell and consider the spaces
\[ P(U) = \mathop{\coprod } _{H \in Ext(U,S^1)} BPU(H)^U \]
where we represent an element of $ Ext(U,S^1) $ by the essentially unique Hilbert bundle $H$ with a stable projective representation
of $ \G \rtimes U $ inducing the extension. 

In fact, $P$ is an $O_{\G}$-space. Now, we can use theorem 6.3 to construct the $\G$-space $C(P)$. This space satisfies
$ C(P)^U \simeq P(U) $ for every $\G$-orbit $U$. Also, it carries a tautological $\G$-stable projective bundle, and so
we have a $\G$-map $ C(P) \to Map(E\G,BPU(H)) $ into the space that represents the functor $ X \to H_{\G}^3(X,\Z ) $.
This map induces an isomorphism
\[ [ X , C(P) ]_{\G} \longrightarrow H_{\G}^3(X,\Z ) \]
it is enough to check the cases $ X = U \times S^i $, where $U$ is a $\G$-orbit. In fact, since finite $G$-CW-complexes are
built out of spaces of the form $ G/H \times S^i $, where $H$ is a closed subgroup of $G$, it is enough to check this for
the cases $ X = U \times S^i $, where $ \G \rtimes U $ is weakly equivalent to $ G \rtimes G/H $.  But this reduces to proving
the isomorphism $ \pi _i(P(U)) \cong H^{3-i}(B(\G \rtimes U),\Z ) $, which follows from the diagram:
\diagram
\pi _i(P(U)) & \rTo & H^{3-i}(B(\G \rtimes U),\Z ) \\
\dTo^{\cong} & & \dTo^{\cong} \\
\pi _i(P_H) & \rTo^{\cong} & H^{3-i}(BH, \Z ) \\
\enddiagram
where the map in the bottom row is an isomorphism by the results in \cite{AS2}.

\end{enumerate}

\end{proof}


\begin{thebibliography}{99}

\bibitem{ALR} A. Adem, J. Leida and Y. Ruan, \emph{Orbifolds and
stringy topology}, Cambridge Tracts in Mathematics, Volume 171, 2007.

\bibitem{AS} M. Atiyah and G. Segal, \emph{Equivariant $K$-theory and
completion}, J. Diff. Geom. 3, 1969, 1-18.

\bibitem{AS2} M. Atiyah and G. Segal,
\emph {Twisted $K$-theory},  Ukr. Mat. Visn.  1, 2004,  no. 3, 287-330.

\bibitem{B} E.H. Brown, \emph{Cohomology theories}, Ann. of Math. 
Vol. 75, No. 3, May 1962, 467-484.

\bibitem{C} J. Cantarero, \emph{Equivariant $K$-theory, groupoids and proper actions}, 
arXiv:0803.3244, 2008.

\bibitem{D} C. Dwyer, \emph{Twisted equivariant $K$-theory for proper actions
of discrete groups}, Ph.D. Thesis, 2005.

\bibitem{E} A. D. Elmendorf, \emph{Systems of fixed point sets}, 
Trans. Amer. Math. Soc. 277 (1983), 275–284.

\bibitem{EM} H. Emerson, R. Meyer, \emph{Equivariant representable
$K$-theory}, arXiv:0710.1410v1, 2007.

\bibitem{FHT} D. Freed, M. Hopkins, C. Teleman,
\emph{Twisted $K$-theory and Loop Group Representations},
arXiv:math/0312155, 2003.

\bibitem{FHT2} D. Freed, M. Hopkins, C. Teleman,
\emph{Loop groups and twisted $K$-theory II}, arXiv:math/0511232, 2005.

\bibitem{FHT3} D. Freed, M. Hopkins, C. Teleman,
\emph{Twisted equivariant $K$-theory with complex coefficients},
 J. Topol.  1, 2008,  no. 1, 16--44.

\bibitem{GH} D. Gepner and A. Henriques, \emph{Homotopy theory of orbispaces},
arXiv:math.AT/0701916, 2007.

\bibitem{J} S. Jackowski, \emph{Families of subgroups and completions},
J. Pure Appl. Algebra 37, 1985, 167-179. 

\bibitem{K} N. Kitchloo, \emph{Dominant $K$-theory and integrable highest
weight representations of Kac-Moody groups}, arXiv:math/0710.0167v1, 2007.

\bibitem{La} A. Lahtinen, \emph{The Atiyah-Segal completion theorem in twisted K-theory},
arXiv:0809.1273, 2009.

\bibitem{Lu} W. Luck, \emph{Survey on classifying spaces for families of subgroups}, Infinite groups: geometric, 
combinatorial and dynamical aspects,  269--322, Progr. Math., 248, Birkhäuser, Basel, 2005.

\bibitem{LO} W. Luck and B. Oliver, \emph{The completion theorem in $K$-theory for
proper actions of a discrete group}, Topology 40, 2001, 585-616.

\bibitem{M} J.P. May, \emph{Classifying spaces and fibrations}, 
Mem. Amer. Math. Soc. No. 155, 1975.

\bibitem{MM} I. Moerdijk and J. Mrcun, \emph{Introduction to foliations and Lie groupoids},
Cambridge University Press, 2003.

\bibitem{P} Alan L.T. Paterson, \emph{Groupoids, inverse semigroups, and their
operator algebras}, Progress in mathematics, vol. 170, Birkhauser, Boston, 1998.

\bibitem{Ph} N.C. Phillips, \emph{Equivariant $K$-theory for proper actions}, Pitman
research notes in mathematics, vol. 178, 1989.

\bibitem{Ph2} N.C. Phillips, \emph{Equivariant $K$-theory for proper actions II: Some cases
in which finite dimensional bundles suffice}, Index theory of elliptic operators, foliations
and operator algebras, Contem. Math. 70, 1988, 205-227. 

\bibitem{S} J. Sauer, \emph{K-theory for proper smooth actions of totally disconnected groups}, 
High-dimensional manifold topology, World Sci. Publ., River Edge, NJ, 2003, 427-448.

\bibitem{SG1} G. Segal, \emph{Classifying spaces and spectral sequences}, 
Inst. Hautes Etudes Sci. Publ. Math. 1968, no. 34, 105-112.

\bibitem{SG2} G. Segal, \emph{The representation ring of a compact Lie group},
Inst. Hautes Etudes Sci. Publ. Math. 1968, no. 34, 113-128.

\bibitem{SG3} G. Segal, \emph{Equivariant $K$-theory},
Inst. Hautes Etudes Sci. Publ. Math. 1968, no. 34, 129-151.

\end{thebibliography}
\end{document}